\documentclass{amsart}
\usepackage{amssymb,amsmath,amsthm}

\usepackage{enumitem}
\usepackage{microtype}
\usepackage{color}

\usepackage{amsfonts}
\usepackage{mathtools}

\usepackage[]{hyperref}

\usepackage{cite}

\numberwithin{equation}{section}

\def\pa{\partial}

\newcommand{\R}{\mathbb{R}}

\let\Re\relax
\let\Im\relax
\DeclareMathOperator{\Re}{\text{Re}}
\DeclareMathOperator{\Im}{\text{Im}}

\newtheorem{theorem}{Theorem}[section]

\newtheorem{lemma}[theorem]{Lemma}

\theoremstyle{definition}

\newtheorem{remark}[theorem]{Remark}

\newcommand{\beq}{\begin{equation}}
\newcommand{\eeq}{\end{equation}}

\newcommand{\beqq}{\begin{equation*}}
\newcommand{\eeqq}{\end{equation*}}

\theoremstyle{remark}

\makeatletter
\newcommand{\Extend}[5]{\ext@arrow0099{\arrowfill@#1#2#3}{#4}{#5}}
\makeatother

\def\leq{\leqslant}
\def\geq{\geqslant}
\def\le{\leqslant}
\def\ge{\geqslant}

\begin{document}
\hspace{1em}
\title[Schr\"odinger equation]{ \bf Long time behaviors for the inhomogeneous NLS with a potential in $\mathbb{R}^3$}
\author{Fanfei Meng}
\address{Qiyuan Lab, Tsinghua University, Beijing 100095, P.R. China.}
\email{mengfanfei@qiyuanlab.com}
\author{Sheng Wang}
\address{Shanghai Center for Mathematical Sciences, Fudan University, Shanghai 200433, P.R. China.}
\email{19110840011@fudan.edu.cn}
\author{Chengbin Xu}
\address{School of Mathematics and Statistics, Qinghai Normal University, Xining, Qinghai 810008, P.R. China.}
\email{xuchengbin19@gscaep.ac.cn}

\begin{abstract}
In this article, we aim to study the scattering of the solution to the focusing inhomogeneous nonlinear Schr\"odinger equation with a potential of form
\begin{align*}
  i\partial_t u+\Delta u- Vu=-|x|^{-b}|u|^{p-1}u
\end{align*}
in the energy space $H^1(\R^3)$. We prove a scattering criterion, and then we use it together with Morawetz estimate to show the scattering theory, which generalizes the results of Dinh \cite{DD} to the non-radial symmetric case.
\\[0.6em]
\\[0.6em]

\end{abstract}
\maketitle
\section{Introduction}\label{section-intro}
We consider the Cauchy problem of the inhomogeneous nonlinear Schr\"odinger equation
with a potential:
 \begin{align}\label{INLC}
 \begin{cases}
   &i\pa_tu+\Delta u- Vu=-|x|^{-b}|u|^{p-1}u,\ \ \ \  (t,x)\in \R\times\R^3 \\
   &u(0,x)=u_0(x) \in H^1(\R^3),
 \end{cases}
 \tag{\text{$\rm{INLS_V}$}}
 \end{align}
where $u:\R\times\R^3\rightarrow\mathbb{C}$, $p>1$, $0<b<1$,
and $V$ is a real-valued potential satisfying
\begin{align}\label{ASU1}
 V \in {\mathcal{K}_0}\cap L^{\frac{3}{2}},
\end{align}
and
 \begin{align}\label{ASU2}
\|V_{-} \|_{\mathcal{K}_0} < 4\pi.
\end{align}
Here, $\mathcal{K}_0$ denotes the closed space
of bounded and compactly supported functions
endowed with the Kato norm
$$ \|V\|_{\mathcal{K}_0} := \sup\limits_{x\in\mathbb{R}^3} \displaystyle \int_{\mathbb{R}^3}\frac{|V(y)|}{|x-y|} {\rm d}y $$
with $V_{-}:=\min\{V,0\}$ being the negative part of $V$.
We recall from \cite{DD,Ho} that
the operator $\mathcal{H}:= -\Delta + V$ has no eigenvalues,
and the dispersive and Strichartz estimates hold
for the corresponding Schr\"odinger flow $\{e^{-it\mathcal{H}}\}$
under the conditions \eqref{ASU1} and \eqref{ASU2}.

It is known that solutions to \eqref{INLC} with suitable regularity
conserve the mass and energy, defined respectively by
$$ M(u) := \int_{\R^3} |u(t,x)|^2 {\rm d}x = M(u_0), $$
$$ E(u) := \displaystyle \frac12 \int_{\R^3} |\nabla u(t,x)|^2 {\rm d}x + \displaystyle \frac12 \int_{\R^3} V(x)|u(t,x)|^2 {\rm d}x -\displaystyle \frac1{p+1} \int_{\R^3} \frac{|u(t,x)|^{p+1}}{\vert x \vert^b} {\rm d}x = E(u_0). $$

Firstly, let us recall the existence results for
the inhomogenous  Schr\"odinger equation without potentials:
\begin{align}
\begin{cases}
&i\pa_tu+\Delta u=\lambda|x|^{-b}|u|^{p-1}u,\ \ \ \  (t,x)\in \R\times\R^{d} \\
&u(0,x)=u_0(x),
\end{cases}
\label{INLS}
\tag{\text{$\rm{INLS}$}}
\end{align}
where $\lambda=1$(resp. $\lambda=-1$) corresponds to the defocusing(resp. focusing) case.
Equation \eqref{INLS} is invariant under the scaling
$$u_{\lambda}(x,t):= \lambda^{\frac{2-b}{p-1}} u(\lambda x, \lambda^2t), \quad \lambda > 0$$
which preserves the
 $\dot H^{s_c}(\R^d)$-norm with $s_c := \frac{d}{2} - \frac{2-b}{p-1}.$
Thus, the equation \eqref{INLS} is called $\dot H^{s_c}(\R^d)$-critical.
The case where $s_c=1$(resp. $s_c=0$) is called energy(resp. mass)-critical,
and the case where $0<s_c <1$ is called intercritical.

Genoud-Stuart \cite{GS} firstly obtained the $H^1$ well-posedness
by using an abstract theory of Cazenave \cite{Ca} for $d\geq1$ and $0<b<\min\{2,d\}$.
Later, Guzm\'an \cite{MG} used Strichartz estimates and the standard Picard iterative technique to revisit the local well-posedness in energy space $H^1$ with $d\geq2, 0<b<b^*$($b^*=\frac{d}3, d=2, 3$ or $b^*=2$ for $d\geq4$). Dinh \cite{DV} extended the results of  Guzm\'an in dimension three for $0<b<\frac32$ and $\frac{7-2b}{3}<p<\frac{5-2b}{2b-1}$.

In the focusing case where $\lambda =-1$,
equation \eqref{INLS} admits a global but non-scattering solution
$$u(t,x)=e^{it}Q(x), $$
where $Q$ is the ground state,
that is, the unique positive radial solution to the elliptic equation (see \cite{CC, DD})
\begin{align}\label{Var-1}
\Delta Q-Q+|x|^{-b}|Q|^{p-1}Q=0.
\end{align}
Genoud \cite{Gen} considered \eqref{INLS} of mass critical and proved the global well-posedness for the $H^1$ initial date
below the ground state, i.e., $\|u_0\|_{L^2}\le\|Q\|_{L^2}$.
In \cite{Far}, Farah proved the global well-posed in $H^1$ below the ground state in the intercritical case.
Regarding the scattering,
Farah-Guzm\'an \cite{FG1} proved the energy scattering of radial solution
when $0<b<\frac{1}{2}$, $p=3$ and  $d=3$.
It was extended by
Miao-Murphy-Zheng \cite{MMZ} to the non-radial setting,
based on the concentration-compactness method developed by Kenig-Merle \cite{KM}.
Moreover,
Xu-Zhao \cite{XZ} proved the energy scattering with $0<b<1$ and $p>3-b$,
by adapting a new argument of Arora-Dodson-Murphy \cite{ADM} in dimension two.
We also refer to the works by Dinh \cite{Din}
and Farah\cite{Far} for the finite-time blow-up solutions.

Secondly, we turn to the development of nonlinear Schr\"odinger equations with potential:
\begin{align}\label{NLS_V}
	 \begin{cases}
	&i\pa_tu+\Delta u- Vu=-|u|^{p-1}u,\ \ \ \  (t,x)\in \R\times\R^d \\
	&u(0,x)=u_0(x),
	\end{cases}
	\tag{\text{$\rm{NLS_V}$}}
\end{align}
whose scattering behavior of solutions has been also studied.
Hong \cite{Ho} proved the energy scattering in the case where $p=3$,
by applying the concentration-compactness method.
And Hamano-Ikeda \cite{HI} extended it to the whole subcritical regime for $V$ satisfying \eqref{ASU1} and \eqref{ASU2}.
Recently, Miao-Murphy-Zheng\cite{MMZ2} studied
the threshold scattering for this model with a repulsive potential ($V$ in \eqref{NLS_V} satisfies \eqref{ASU1}, $V\geq0$ and $x\cdot\nabla V\leq 0$).
We also mention that, Miao-Zhang-Zheng\cite{MMZ} studied the Cauchy problem with Coulomb potential ($V=\frac{C}{|x|}$ in \eqref{NLS_V}) where $1<p\le5$ and $d=3$.
Moreover, the focusing \eqref{NLS_V} with inverse square potential ($V=\frac{a}{|x|^2}$ and $a\geq0$) was studied by Zheng \cite{Zh}.

Thirdly, there are also several works for the present \eqref{INLC}.
Deng-Lu-Meng \cite{DLM} studied the blow-up versus global well-posedness of \eqref{INLC} with inverse-square potential $ V = \frac{a}{\vert x \vert^2} $ for $ a > -\frac{(d-2)^2}4 $.
Dinh \cite{DD} and Guo-Wang-Yao \cite{GWY} proved
the global dynamics for a class $L^2$-supercritical nonlinear inhomogeneous Schr\"odinger equation
with a potential in dimension three with radially symmetric initial data.
Inspired by Campos and Cardoso \cite{CC} and Murphy \cite{Mu},
we study the scattering theory of \eqref{INLC} in the non-radial setting here.

To show the scattering theory of Theorem \ref{MR}, our strategy in this paper is to divide the whole proof into two steps.
And the first step is to establish a Marowetz type estimate stated as in Lemma \ref{de} and the second step is the scattering criterion as follows.

\begin{theorem}[Scattering criterion]\label{SC}
For $0<b<1$, $1+\frac{4-2b}{3}<p<1+4-2b$, and $V:\R^3\rightarrow\mathbb{R}$
satisfies \eqref{ASU1}, \eqref{ASU2} in \eqref{INLC}.
If $u \in H^1$ is a solution to \eqref{ASU1} defined on $[0,+\infty)$ with the following priori bound
	\begin{align}\label{pb}
		\sup\limits_{t} \|u(t)\|_{H^1}:=E < +\infty,
	\end{align}
then there exist two constants $R>0$ and $\epsilon >0$,
depending only on $E$, $p$ and b (but not on $u$ or $t$ ),
such that if
    \begin{align}\label{ASU3}
   	 \liminf\limits_{t \to +\infty } \displaystyle \int_{B(0,R)} |u(t,x)|^2 {\rm d}x \leq \epsilon^2,
   \end{align}
then u scatters forward in time in $H^1(\mathbb{R}^3)$,
that is,
there exists a function $u_+ \in H^1$ such that
$$\liminf_{t \to +\infty} \|u(t)-e^{-it\mathcal{H}}u_+\|_{H^1} =0.$$	
\end{theorem}

Next, we formulate the scattering result
for the solutions to \eqref{INLC}.
We shall use the Generalized Sobolev norms
\begin{align}
\|\Gamma f\|^2_{L^2}:= \displaystyle \int_{\R^3} |\nabla f|^2 {\rm d}x + \displaystyle \int_{\R^3} V|f|^2 {\rm d}x.
\end{align}

\begin{theorem} [Scattering] \label{MR}
Consider equation \eqref{INLC} with $0<b<1$ and $1+\frac{4-2b}{3}<p<1+4-2b.$
Assume that $V:\R^3\rightarrow\mathbb{R}$ satisfies \eqref{ASU1},
and in addition $V\ge 0$, $x\cdot\nabla V \leq 0$,
$x\cdot \nabla V \in L^{r}$, where $r\in [\frac{3}{2},\infty)$.
Let $u_0\in H^1$ satisfy
\begin{align}\label{q1}
	E(u_0)[M(u_0)]^{\frac{1-s_c}{s_c}}<\|\nabla Q\|_{L^2}\|Q\|^{\frac{1-s_c}{s_c}}_{L^2}.
\end{align}
Then, if
\begin{align}\label{q2}
	\|\Gamma u_0\|_{L^2}\|u_0\|^{\frac{1-s_c}{s_c}}_{L^2}<\|\nabla Q\|_{L^2}\|Q\|^{\frac{1-s_c}{s_c}}_{L^2},
\end{align}
we have that the corresponding solution $u$ to \eqref{INLC} exists globally in time
and satisfies
\begin{align} \label{q3}
\|\Gamma u(t)\|_{L^2}\|u(t)\|^{\frac{1-s_c}{s_c}}_{L^2}<\|\nabla Q\|_{L^2}\|Q\|^{\frac{1-s_c}{s_c}}_{L^2}
\end{align}
for all $t\in \mathbb{R}$. Moreover, the global solution scatters in $H^1$ in both time directions.
	
\end{theorem}

\begin{remark}
In fact, the Theorem \ref{MR} not only extends  the results of Dinh \cite{DD} to the  non-radial case but also extends the integrability condition of $x\cdot \nabla V\in L^r$ to any $r\in[\frac{3}{2}, \infty)$. Because of the absence of translation symmetry in the potential, the non-radial case cannot be achieved directly by the  method of Miao--Murphy--Zheng \cite{MMZ} and Dodson--Murphy \cite{DM}. Here, we construct a special class of cut off functions whose derivatives are bounded and do not depend on compact set $B_R(0)$, and use the decay of the factor $|x|^{-b}$ in the nonlinear term  which can replace the radial Sobolev embedding in some sense. Based on this, we establish the non-radial scattering criterion  and  more complicated Morawetz estimate with potential  for non-radial. In the Morawetz estimate, we find that the error term
\end{remark}

\begin{align*}
  \left| \int_{|x| \ge \frac{R}2} \nabla a \cdot \nabla V |u|^2 {\rm d}x \right| \leq C\|x\cdot\nabla V\|_{L^{r}(|x|\geq \frac{R}2)}\|u\|^{2}_{L^{2r'}}
\end{align*}
tends to $0$ as $ R \to \infty $. From this, we can extend the integrability of  $x\cdot\nabla V\in L^r(\R^3)$ to $r\in[\frac32,\infty)$ and generalize the results of Dinh \cite{DD}  to the non-radial symmetric case.

{\bf Outline.}
This paper is organized as follows.
First, in Sect.~$2$ we recall some preliminaries for the Schr\"odinger operator.
Then, in Sect.~$3$, we recall the variational analysis of the ground state.
In Sect.~$4$,  we prove the scattering criterion in Theorem \ref{SC}.
Finally, in Sect.~$5$, we apply the scattering criterion,
together with the Virial-Morawetz estimate, to prove Theorem \ref{MR}.
\medskip

\section{Preliminaries}
{\bf Notations.}
We use the notation
$X\lesssim Y$ to denote $X\leq CY$ for some constant $C>0$ and use the notation $ X \sim Y $ to denote $ X \lesssim Y $ and $ Y \lesssim X $ at the same time.
Similarly, $X\lesssim_{u} Y$ means that there exists a constant $C:=C(u)$ depending on $u$ such that $X\leq C(u)Y$.
We also use the notation $A=\mathcal{O}(B)$,
which means that there exists a constant $l>0$ such that $ \lim_{B \to 0} = l $.
The derivative operator $\nabla$ refers to the spatial  variable only.
Let $L^r(\mathbb{R}^3)$ denote the usual space of integrable functions
$f:\mathbb{R}^3\rightarrow\mathbb{C}$ endowed with the norm
$$\|f\|_r := \|f\|_{L^r}=\Big(\int_{\mathbb{R}^3}|f(x)|^r {\rm d}x\Big)^{\frac1r},$$
with the usual modification when $r=\infty$.
For any non-negative integer $k$,
$H^{k,r}(\mathbb{R}^3)$ denotes the Sobolev space,
defined as the closure of smooth compactly supported functions
under the norm
\begin{align*}
\|f\|_{H^{k,r}} := \sum_{|\alpha|\leq k} \Big\| \frac{\partial^{\alpha}f}{\partial x^{\alpha}} \Big\|_{r}.
\end{align*}
In particular, when $ r = 2 $, we denote it by $ H^k $. 
Besides, we use the norm
\begin{align*}
\| f \|_{\dot{H}^s} := \left(\int_{\R^3} | \xi |^{2s} | \hat{f} \left( \xi \right)|^2 {\rm d}\xi \right)^{\frac12}
\end{align*}
to denote the homogeneous Sobolev space $ \dot{H}^s $, where $ s $ is a real number.
And for any time slab $ I $,
$ L_t^q(I, L_x^r(\mathbb{R}^3)) $ means the space of $ L_x^r $-valued functions with the norm
\begin{align*}
\|f\|_{L_{t}^qL^r_x(I\times \R^3)} = \bigg( \int_{I}\|f(t,x)\|_{L^r_x}^q {\rm d}t \bigg)^\frac{1}{q},
\end{align*}
with the usual modifications when $ q $ or $ r $ is infinite.
For simplicity, we also use the shorthand
$ \| f \|_{L^q(I, L^r(\R^3))} $ or $ \| f \|_{L^q L^r(I \times \mathbb{R}^3)} $.

\noindent
\subsection{Strichartz estimates}
Let us start this section by introducing the notation used throughout the paper. We recall some Strichartz estimates associating with the linear Schr\"odinger propagator.

We say the pair $(q,r)$ is $L^2-$adimissible or simply admissible pair if $(q,r)$ satisfies the condition
$$\frac{2}{q}=\frac{3}{2}-\frac{3}{r},\ \  2\leq r\leq 6. $$
 The pair is called $\dot{H}^s-$admissible, if
$$\frac{2}{q}=\frac{3}{2}-\frac{3}{r}-s.$$
And $(q,r)$ is called $\dot{H}^{-s}-$admissible, if
$$\frac{2}{q}=\frac{3}{2}-\frac{3}{r}+s.$$

 Given $s\in (0,1)$, let
\begin{equation}
\Lambda_s=\left \{ (q,r) \quad \text {is} \quad \dot{H}^s- \text{admissible} \bigg| \left(\frac{6}{3-2s}\right)^+ \le r\le 6^-\right \}
\end{equation}
and
\begin{equation}
	\Lambda_s=\left \{ (q,r) \quad  \text {is} \quad  \dot{H}^{-s}- \text{admissible} \bigg| \left(\frac{6}{3-2s}\right)^+ \le r\le 6^-\right \},
\end{equation}
where the notations $a^+$ and $a^-$ denote, respectively, $a+\epsilon$ and $a-\epsilon$, for fixed $0<\epsilon\ll1$.

We define the following Strichartz norm
$$\|u\|_{S(\dot{H}^s,I)}=\sup_{(q,r)\in \Lambda_s}\|u\|_{L_t^qL_x^r(I)}$$
and dual Strichartz norm
$$\|u\|_{S^{'}(\dot{H}^{-s},I)}=\inf_{(q,r)\in \Lambda_{-s}}\|u\|_{L_t^{q^{'}}L_x^{r^{'}}(I)}.$$

If $I=\R$, $I$ is omitted usually. For more details see \cite{Ca} and \cite{KT}.

\begin{lemma}[Dispersive estimate and Strichartz estimate, \cite{Ho}] \label{Dispersive and Strichartz }
The following statements hold:
(i) Let $V:\R^3\rightarrow\mathbb{R}$ satisfy \eqref{ASU1} and \eqref{ASU2}. Then it holds that
\begin{align}\label{dis}
\|e^{-it\mathcal{H}}\|_{L^1\rightarrow L^{\infty}} \lesssim |t|^{-\frac{3}{2}}.
\end{align}
(ii) Let $V:\R^3\rightarrow\mathbb{R}$ satisfy \eqref{ASU1} and \eqref{ASU2}. Then it holds that
\begin{align}
\|e^{-it\mathcal{H}}f\|_{S(\dot{H}^s)}\lesssim \|f\|_{\dot{H}^s} \quad s\ge 0,
\end{align}
and
\begin{align}
\left\|\int_{0}^{t}e^{-i(t-s)\mathcal{H}}g(\cdot,s) {\rm d}s\right\|_{S(\dot{H}^s)} \lesssim \|g\|_{S^{'}(\dot{H}^{-s})} \quad s\in(0,1).
\end{align}
\end{lemma}

\subsection{Equivalence of Sobolev norms}
 In this subsection, we define the homogeneous and inhomogeneous Sobolev spaces associated to $\mathcal{H}$ as the closure of $C^{\infty}_0(\mathbb{R}^3)$ under the norms
 $$\|f\|_{\dot{W}_V^{\gamma,r}}:=\|\Gamma^{\gamma}f\|_{L^r},\quad \|f\|_{W_V^{\gamma,r}}:=\|\langle\Gamma\rangle^{\gamma}f\|_{L^r},\quad \Gamma:= \sqrt{\mathcal{H}}.$$
 To simplify the notation, we denote $\dot{H}_V^{\gamma}:=\dot{W}_V^{\gamma,2}$ and $H_V^{\gamma}:=W_V^{\gamma,2}$.

\begin{lemma} [Sobolev inequalities, \cite{Ho}]
Let $V:\R^3\rightarrow\mathbb{R}$ satisfy \eqref{ASU1} and \eqref{ASU2}. Then it holds that	
$$ \|f\|_{L^q} \lesssim \|f\|_{\dot{W}_V^{\gamma,r}},\quad  \|f\|_{L^q} \lesssim \|f\|_{W_V^{\gamma,r}},$$
where $1<p<q<\infty$, $1<p<\frac{3}{\gamma}$, $0\leq\gamma\leq 2$ and $\frac{1}{q}=\frac{1}{p}-\frac{\gamma}{3}$.

\end{lemma}

\begin{lemma}[Equivalence of Sobolev spaces, \cite{Ho}]
Let $V:\R^3\rightarrow\mathbb{R}$ satisfy \eqref{ASU1} and \eqref{ASU2}. Then it holds that	
$$\|f\|_{\dot{W}_V^{\gamma,r}}\sim\|f\|_{\dot{W}^{\gamma,r}},\quad \|f\|_{W_V^{\gamma,r}}\sim\|f\|_{W^{\gamma,r}},$$
where $1<r<\frac{3}{\gamma}$ and $0\leq\gamma\leq2$.
\end{lemma}
\subsection{Some nonlinear estimates}
We recall some interpolation estimates for nonlinearities, which plays an important role in proving scattering theory.
\begin{lemma}[Nonlinear estimates, \cite{Cam, MG}]\label{nle}
Let $u, v \in C_{0}^{\infty}\left(\mathbb{R}^{3+1}\right)$,
$1+\frac{4-2 b}{3}<p<1+\frac{4-2 b}{3-2}$ and $0<b<1$.
Then there exists $0 \leq \theta=\theta( p, b) \ll p-1$ such that the following estimates hold
\begin{align}
\left\||x|^{-b}|u|^{p-1} v\right\|_{S^{\prime}\left(\dot{H}^{-s_{c}}, I\right)} & \lesssim\|u\|_{L_{t}^{\infty} H_{x}^{1}}^{\theta}\|u\|_{S\left(\dot{H}^{s c}, I\right)}^{p-1-\theta}\|v\|_{S\left(\dot{H}^{s_{c}}, I\right)} ,\\
\left\||x|^{-b}|u|^{p-1} u\right\|_{S^{\prime}\left(L^{2}, I\right)} & \lesssim\|u\|_{L_{t}^{\infty} H_{x}^{1}}^{\theta}\|u\|_{S\left(\dot{H}^{s_{c}}, I\right)}^{p-1-\theta}\|u\|_{S\left(L^{2}, I\right)} ,\\
\left\|\nabla\left(|x|^{-b}|u|^{p-1} u\right)\right\|_{S^{\prime}\left(L^{2}, I\right)} & \lesssim\|u\|_{L_{t}^{\infty} H_{x}^{1}}^{\theta}\|u\|_{S\left(\dot{H}^{s_{c}}, I\right)}^{p-1-\theta}\|\nabla u\|_{S\left(L^{2}, I\right)} ,\\
\left\||x|^{-b}|u|^{p-1} u\right\|_{L_{I}^{\infty} L_{x}^{r}} & \lesssim\|u\|_{L_{I}^{\infty} H_{x}^{1}}^{p}
\end{align}
for $\frac{2(3-b)}{3+4-2 b}<r<\frac{2(3-b)}{3+2-2 b}$.
\end{lemma}

\subsection{Local well-posedness in $H^1$}
In this subsection, we recall the local well-posedness in $H^1$,
the global well-posedness of small initial date and its corresponding scattering theory under the assumptions \eqref{ASU1} and \eqref{ASU2}.
Particularly, we denote the $\Lambda_0$ the Strichartz norm for any time interval $I\subset \mathbb{R}$
\begin{align}
	\|u\|_{S(L^2,I)}:=\sup_{(q,r)\in \Lambda_0}\|u\|_{L^q(I,L^r)},\quad \|u\|_{S^{'}(L^2,I)}:=\inf_{(q,r)\in \Lambda_0}\|u\|_{L^{q^{'}}(I,L^{r^{'}})}.
\end{align}
where $(q,q^{'})$ and $(r,r^{'})$ are H\"older's conjugate pairs. The condition $2\leq r <3$ can ensure $\dot{W}_V^{1,r}\sim \dot{W}^{1,r}$, and $\dot{W}_V^{1,r^{'}}\sim \dot{W}^{1,r^{'}}$.

\begin{theorem}[Local well-posedness \cite{CC, DD}]\label{LWP}
Let $0<b<1$ and $1+\frac{4-2b}{3}<p<1+4-2b$. Let $V:\R^3\rightarrow\mathbb{R}$ satisfy \eqref{ASU1} and \eqref{ASU2}. Then the equation is locally well-posed in $H^1$.
\end{theorem}

\begin{theorem}[Small initial date, \cite{DD,CC}] \label{SDGWP}
Let $0<b<1$ and $1+\frac{4-2b}{3}<p<1+4-2b$. Let $V:\R^3\rightarrow\mathbb{R}$. Suppose $\|u_0\|_{H^1}\leq E$. Then there exists $\delta_{sc}=\delta_{sc}(E)>0$ such that if
$$\|e^{-it\mathcal{H}} u_0\|_{S(\dot{H}^{s_c},[0,+\infty))}\leq \delta_{sc},$$
then there exists a unique global solution $u$ to \eqref{INLC} with initial date $u_0$ satisfying
$$\|u\|_{S(\dot{H}^{s_c},[0,+\infty))} \leq \|e^{-it\mathcal{H}} u_0\|_{S(\dot{H}^{s_c},[0,+\infty))}.$$

Furthermore,  $u$ scatters forward in time in $H^1$, i.e, there exists $u_+\in H^1$ such that
$$\lim \limits_{t\to \infty }\|u(t)-e^{-it\mathcal{H}}\|_{H^1}=0.$$
\end{theorem}


\section{Variational analysis}
Let us recall some properties related to the ground state $Q$ which is the unique positive radial decreasing solution to the elliptic equation (for more detail, see \cite{CC, DD})
$$\Delta Q-Q+|x|^{-b}|Q|^{p-1}Q=0.$$

\begin{lemma}[Coercivity I,\cite{Cam, CC, DD}]
	Let $0<b<1$ and $1+\frac{4-2b}{3}<p<1+4-2b.$ Let $V:\R^3\rightarrow\mathbb{R}$
	satisfy \eqref{ASU1}, $V\ge 0$. Let $u_0\in H^1$ satisfy \eqref{q1} and \eqref{q2}, then the corresponding solution to the focusing problem \eqref{INLC} satisfies
	\begin{align}
	\|\Gamma u(t)\|_{L^2}\|u(t)\|^{\frac{1-s_c}{s_c}}_{L^2}<\|\nabla Q\|_{L^2}\|Q\|^{\frac{1-s_c}{s_c}}_{L^2}, \qquad \forall ~ t \in \mathbb{R}.
	\end{align}
	In particular, the corresponding solution to the focusing problem \eqref{INLC} exists globally in time. Moreover, there exists $\delta=\delta(u_0,Q)>0$ such that
	\begin{align}
	\|\Gamma u(t)\|_{L^2}\|u(t)\|^{\frac{1-s_c}{s_c}}_{L^2}<(1-2\delta)\|\nabla Q\|_{L^2}\|Q\|^{\frac{1-s_c}{s_c}}_{L^2}, \qquad \forall ~ t \in \mathbb{R}.
	\end{align}
	
\end{lemma}

\begin{remark}
	By the assumption $V\geq 0$ and the same argument as above, we see that if  $u_0\in H^1$ satisfies \eqref{q1} and \eqref{q2}, then the corresponding solution to the focusing problem \eqref{INLC} satisfies
	\begin{align}
	\|\nabla u(t)\|_{L^2}\|u(t)\|^{\frac{1-s_c}{s_c}}_{L^2}<\|\nabla Q\|_{L^2}\|Q\|^{\frac{1-s_c}{s_c}}_{L^2}, \qquad \forall ~ t \in \mathbb{R}.
	\end{align}
	In particular, the corresponding solution to the focusing problem \eqref{INLC} exists globally in time. Moreover, there exists $\delta=\delta(u_0,Q)>0$ such that
	\begin{align}\label{qz}
	\|\nabla u(t)\|_{L^2}\|u(t)\|^{\frac{1-s_c}{s_c}}_{L^2}<(1-2\delta)\|\nabla Q\|_{L^2}\|Q\|^{\frac{1-s_c}{s_c}}_{L^2}, \qquad \forall ~ t \in \mathbb{R}.
	\end{align}
\end{remark}

\begin{lemma} [Coercivity II, \cite{Cam, CC,DD}]  \label{g}
	Suppose $f\in H^1(\mathbb{R}^3)$, that
	\begin{align}
	\|f\|^{\frac{1-s_c}{s_c}}_{L^2}\|\nabla f\|_{L^2}\le (1-\delta)\|Q\|^{\frac{1-s_c}{s_c}}_{L^2}\|\nabla Q\|_{L^2}.
	\end{align}
	then there exists $\delta^{'}=\delta^{'}(\delta)>0$  so that
	\begin{align}
	\displaystyle \int_{\R^3} |\nabla f|^2 {\rm d}x + \left(\frac{3-b}{p+1}-\frac{3}{2}\right) \int_{\R^3} \frac{|f|^{p+1}}{\vert x \vert^b} {\rm d}x \ge \delta^{\prime} \int_{\R^3} \frac{|f|^{p+1}}{\vert x \vert^b} {\rm d}x. 
	\end{align}
\end{lemma}

\begin{lemma}[Kinetic energy on balls, \cite{Cam, CC}]\label{ge}
	Let $\chi$ be a smooth cutoff to the set $\{|x|\le1/2\}$ and set $\chi_R(x):=\chi(\frac{x}{R})$. If $f\in H^1$, then
	\begin{align}
	\displaystyle \int_{\R^3} |\nabla (\chi_R f)|^2 {\rm d}x = \displaystyle \int_{\R^3} \chi_R^2|\nabla f|^2 {\rm d}x - \displaystyle \int_{\R^3} \chi_R \Delta(\chi_R)|f|^2 {\rm d}x. 
	\end{align}
	In particular,
	\begin{align}
	\left|\displaystyle \int_{\R^3} |\nabla(\chi_Rf)|^2 {\rm d}x - \displaystyle \int_{\R^3} \chi_R^2|\nabla f|^2 {\rm d}x \right| \leq \frac{c}{R^2} \| f \|_{L^2}^2.
	\end{align}
\end{lemma}

\begin{lemma}[Local coercivity, \cite{Cam, CC, DD}]\label{lc}
	let u be a global $H^1$-solution to \eqref{INLC} satisfying  \eqref{qz}. There exists $\bar{R}= \bar{R}(\delta,M[u_0]),Q,s_c) >0$ such that, for any $R\geq \bar{R},$
	\begin{align}
	\|\chi_Ru(t)\|^{\frac{1-s_c}{s_c}}_{L^2}\|\nabla(\chi_Ru(t))\|_{L^2}\leq(1-\delta)\|Q\|^{\frac{1-s_c}{s_c}}_{L^2}\|\nabla Q\|_{L^2}.
	\end{align}
	In particular, by Lemma \ref{g}, there exists $\delta^{'}=\delta^{'}(\delta)>0$ such that
	\begin{align}
	\displaystyle \int_{\R^3} |\nabla(\chi_Ru(t))|^2 {\rm d}x + \left(\frac{3-b}{p+1}-\frac{3}{2}\right) \displaystyle
	\int_{\R^3} \frac{|\chi_Ru(t)|^{p+1}}{\vert x \vert^b} {\rm d}x \geq \delta^{\prime} \displaystyle \int_{\R^3} \frac{|\chi_R u|^{p+1}}{\vert x \vert^b} {\rm d}x. 
	\end{align}
\end{lemma}

\section{Proof of the scattering criterion}
In this section, we will prove Theorem \ref{SC}, which reduces the proof of Theorem \ref{MR} to the verfication of condition \eqref{ASU3} on local mass in long time.
The history of Theorem \ref{SC} comes back to Dodson-Murphy \cite{DM} for classical nonlinear Schr\"odinger equation in 3D and then F. Meng \cite{M} for the nonlinear Hartree equation in 5D.
In \cite{CC}, it was proved for the \eqref{INLS} without potential.
For the \eqref{INLC} with a potential, it was proved in \cite{DD} under the radial assumption.
Here, we prove the result for non-radial initial data and the full subcritical range in dimension three.

\begin{lemma}\label{3.1}
Let $0<b<1$ and $1+\frac{4-2b}{3}<p<1+4-2b.$ Let $V:\R^3\rightarrow\mathbb{R}$
satisfy \eqref{ASU1}, \eqref{ASU2} and  $u$ be a (possibly non-radial) $H^1$-solution to \eqref{pb}. If $u$ satisfies \eqref{ASU3} for some $0<\epsilon<1$, then there exist $\gamma$, $T>0$ such that
$$\|e^{-i(t-T)\mathcal{H}}u(T)\|_{S(\dot{H}^{s_c},[0,+\infty))}\lesssim \epsilon^{\gamma}.$$
\end{lemma}

\begin{proof}
First, we fix the parameters $\alpha, \gamma >0$ (to be chosen later). By Strichartz estimate, there exists $T_0> \epsilon^{-\alpha}$ such that
\begin{align}\label{es1}
	\|e^{-it\mathcal{H}}u_0\|_{S(\dot{H}^{s_c},[T_0,+\infty))}\leq \epsilon^{\gamma}.
\end{align}
For $ T \ge T_0 $ to be chosen later, define $ I_1 := [T - \epsilon^{-\alpha}, T] $, $ I_2 := [0, T - \epsilon^{-\alpha}] $ and let $ \eta $ denote a smooth, spherically symmetric function which equals 1 on $ B(0, 1/2) $ and 0 outside $ B(0, 1) $. 
Set $ \eta_R(x) := \eta(x/R) $ for any $ R > 0 $. 

From Duhamel's formula
\[
u(T) = e^{iT\mathcal{H}} u_0 - i\int_0^T e^{i(t-s)\mathcal{H}} \frac{|u|^{p-1}u}{\vert x \vert^b} {\rm d}s,  
\]	
we obtain
\[
e^{i(t-T)\mathcal{H}} u(T) = e^{-it\mathcal{H}} u_0 + iG_1+ iG_2, 
\] 
where
\begin{align*}
G_j(t) :=\int_{I_j} e^{i(t-s)\mathcal{H}} \frac{|u|^{p-1}u}{\vert x \vert^b} {\rm d}s, \qquad j = 1, 2.
\end{align*}

{\bf Estimate on $G_1$.}
By hypothesis \eqref{ASU3}, we can fix $T\geq T_0$ such that
\begin{align}
\displaystyle \int_{\R^3} \eta_R(x) |u(T,x)|^2 {\rm d}x \lesssim \epsilon^2.
\label{es2}
\end{align}
Multiplying \eqref{INLC} by $\eta_R \bar{u}$, taking the imaginary part and integrating by parts, we obtain
\[ 
\pa_t \displaystyle \int_{\R^3} \eta_R(x)|u(t,x)|^2 {\rm d}x \lesssim \frac{1}{R}, 
\] 
so that, by \eqref{es2}, for $ t \in I_1 $,
\[ 
\displaystyle \int_{\R^3} \eta_R(x) |u(t,x)|^2 {\rm d}x \lesssim \epsilon^2 + \frac{\epsilon^{-\alpha}}{R}. 
\] 
If $ R \ge \epsilon^{-(\alpha+2)}$, then we have $ \| \eta_R u \|_{L^{\infty}_{I_1} L^2_x}\lesssim \epsilon $. 

We choose $ (s, l) \in \Lambda_{s_c} $ as
\[
s=\frac{4(p-1)(p+1-\theta)}{(p-1)[3(p-1)+2b]-\theta[3(p-1)-4+2b]},\quad l=\frac{3(p-1)(p+1-\theta)}{(p-1)(3-b)-\theta(2-b)}. 
\]
By H\"older's inequality and Sobolev embedding, for $ t \in I_1 $,
\begin{align}
\Big\| \eta_R \frac{|u|^{p-1}u}{\vert x \vert^b} \Big\|_{L_x^{l^{\prime}}} 
\lesssim \|u(t)\|^{\theta}_{H_x^1} \|u(t)\|^{p-1-\theta}_{L_x^l} \|\eta_Ru(t)\|_{L_x^l} 
\lesssim \|\eta_Ru(t)\|_{L_x^l}.
\label{es3}
\end{align}

Now, let $ 0 < \mu < 1 $ satisfy $ \frac1l = \frac{\mu}2 + \frac{1-\mu}6 $, we have 
\begin{align}
\| \eta_R u(t) \|_{L_x^l} 
\leq \| u(t) \|_{L_x^6}^{1-\mu} \| \eta_R u(t) \|_{L_x^2}^{\mu} 
\lesssim \epsilon^{\mu},
\label{es4}
\end{align}
uniformly for $ t\in I_1 $. 
We now exploit the decay of the nonlinearity.
By H\"older's inequality and Sobolev embedding again, for $ R \ge 0 $ sufficiently large (depending on $ \epsilon $) and $ t \in I_1 $,
\begin{align}
\Big\| (1-\eta_R) \frac{|u|^{p-1}u}{\vert x \vert^b} \Big\|_{L_x^{l^{\prime}}} 
\lesssim & ~ \Big\| \frac{|u|^{p-1}u}{\vert x \vert^b} \Big\|_{L_x^{l^{\prime}}(|x| > \frac{R}2)} \nonumber \\
\lesssim & ~ \Big\| \frac1{\vert x \vert^b} \Big\|_{L_x^{l_1}(|x| > \frac{R}2)} \| u \|_{L_x^{\theta l_2}}^{\theta} \| u \|_{L_x^l}^{p-\theta} \nonumber \\
\lesssim & ~ \frac1{R^{b l_1 - 3}} \| u \|_{H_x^1}^p 
\lesssim \epsilon^{\mu},
\label{es5}
\end{align}
where $ l_1 $ and $ l_2 $ are such that $ b l_1 > 3 $, $ 2 < \theta l_2 < \frac{3(p-1)}{2-b} $ and
\[ 
\frac1{l^{\prime}} = \frac1{l_1} + \frac1{l_2} + \frac{p - \theta}l. 
\]

Using the Strichartz estimates, together with estimates \eqref{es3}, \eqref{es4}, and \eqref{es5}, we obtain 
\begin{align}
& ~ \left\| \int_{I_1} e^{i(t-s)\mathcal{H}} \frac{|u|^{p-1}u(s)}{\vert x \vert^b} {\rm d}s \right\|_{S(\dot{H}^{s_c}, [T, \infty))} \nonumber \\
\lesssim & ~ \Big\| \frac{|u|^{p-1}u}{\vert x \vert^b} \Big\|_{S^{\prime}(\dot{H}^{-s_c}, I_1)}
\lesssim \Big\| \frac{|u|^{p-1}u}{\vert x \vert^b} \Big\|_{L_t^{s^{\prime}}(I_1, L_x^{l^{\prime}})} \nonumber \\ 
\lesssim & ~ \Big\| \eta_R \frac{|u|^{p-1}u}{\vert x \vert^b} \Big\|_{L_t^{s^{\prime}}(I_1, L_x^{l^{\prime}})} + \Big\| (1 - \eta_R) \frac{|u|^{p-1}u}{\vert x \vert^b} \Big\|_{L_t^{s^{\prime}}(I_1, L_x^{l^{\prime}})} \nonumber \\
\lesssim & ~ |I_1|^{1/s^{\prime}} \epsilon^{\mu} 
= \epsilon^{\mu - \alpha / s^{\prime}} 
= \epsilon^{\mu/2},
\label{es6}
\end{align}
where we choose $ \alpha := s^{\prime}\mu/2 $.

{\bf Estimate on $G_2$.}
Let $ (q, r) \in \Lambda_{s_c} $. 
Define, for small $ \delta \ge 0 $,
\[ 
\frac1h = \left( \frac1{1-s_c} \right) \left[ \frac1q - \delta s_c \right], 
\] 
and
\[ 
\frac1k = \left( \frac1{1-s_c} \right) \left[ \frac1r - s_c \left( \frac{3 - 2 - 4\delta}6 \right) \right]. 
\] 
We see that $ (h, k) \in \Lambda_0 $. 
By interpolation,
\[ 
\| G_2 \|_{L_t^q([T,+\infty), L_x^r)} 
\lesssim \| G_2 \|_{L_t^h([T,+\infty), L_x^k)}^{1-s_c} \| G_2 \|_{L_t^{\frac1{\delta}}([T,+\infty), L_x^{\frac6{3 - 2 - 4\delta}})}^{s_c}. 
\] 
By the dispersive estimate \eqref{dis} and Lemma \ref{nle}, one can get
\begin{align}
& ~ \| G_2 \|_{L_t^{\frac1{\delta}}([T, \infty), L_x^{\frac6{3-2-4\delta}}(\R^3))} \nonumber \\
\lesssim & ~ \left\| \int_{I_2} | t - s |^{-(1 + 2\delta)} \Big\| \frac{|u|^{p-1} u}{\vert x \vert^b} \Big\|_{L_x^{\frac6{3+2+4\delta}}} {\rm d}s \right\|_{L_t^{\frac1{\delta}}([T, \infty))} \nonumber \\
\lesssim & ~ \| u \|_{L_t^{\infty} H_x^1}^p \left\| \left( t - T + \epsilon^{-\alpha} \right) \right\|_{L_t^{\frac1{\delta}}([T, \infty))}
\lesssim \epsilon^{\alpha \delta}.
\label{G2}
\end{align}
Using Duhamel's principle, rewrite
\[ 
iG_2 = e^{it\mathcal{H}} \left[ e^{-i(T-\epsilon^{-\alpha})\mathcal{H}} u(T-\epsilon^{-\alpha}) - u_0 \right], 
\] 
thus, by the Strichartz estimate and \eqref{G2},
\begin{align}
& ~ \| G_2 \|_{L_t^q([T, +\infty), L_x^r)} \nonumber \\
\lesssim & ~ \left\| e^{-it\mathcal{H}} \left[ e^{-i(T-\epsilon^{-\alpha})\mathcal{H}} u(T-\epsilon^{-\alpha}) - u_0 \right] \right\|_{L_t^h([T, +\infty), L_x^k)}^{1-s_c} \| G_2 \|_{L_t^{\frac1{\delta}}([T, +\infty), L_x^{\frac6{3-2-4\delta}})}^{s_c} \nonumber \\
\lesssim & ~ \|u\|^{1-s_c}_{L^{\infty}_tL^2_x}\|G_2\|^{s_c}_{L_t^{\frac1{\delta}}([T,+\infty), L_x^{\frac6{3-2-4\delta}})}
\lesssim \epsilon^{\alpha\delta}.
\end{align}

Therefore, defining $ \gamma := \min \{ \mu/2, \alpha \delta s_c \} $ and recalling that
\[ 
e^{-i(t-T)\mathcal{H}} u(T) = e^{-it\mathcal{H}} u_0 + iG_1 + iG_2, 
\] 
we have
\[ 
\| e^{-i(t-T)\mathcal{H}} u(T) \|_{S(\dot{H}^{s_c}, [T,\infty))} \lesssim \epsilon^{\gamma}. 
\] 
Then, we complete the proof of Lemma \ref{3.1}.
\end{proof}

\begin{proof}[Proof of Theorem \ref{SC}]
Choose $ \epsilon > 0 $ sufficiently small, by Theorem \ref{LWP} and \ref{SDGWP}, 
\[ 
\| e^{-it\mathcal{H}} u(T) \|_{S(\dot{H}^{s_c}, [0, \infty))} 
= \| e^{-i(t-T)\mathcal{H}} \|_{S(\dot{H}^{s_c}, [T,\infty))} 
\lesssim \epsilon^{\gamma} 
\lesssim \delta_{sc}, 
\] 
where $ \delta_{sc} $ is given in Theorem \ref{SDGWP}.
Thus by small data scattering theory, $ u $ scatters forward in time in $ H^1 $, as desired.
\end{proof}

\section{Proof of Theorem \ref{MR}}
In this section, we will prove the Morawetz estimate, and use it together with scattering criterion to  obtain the main result Theorem \ref{MR}.

\noindent

\begin{lemma}[Virial/Morawetz identity]\label{VMi}
Let $ a : \R^3 \rightarrow \R $ be a real-valued weight.
If $ \|\nabla a\|_{L^{\infty}} < \infty $, define
\[
Z(t) = 2 \Im \int_{\R^3} \bar{u} \nabla u \cdot \nabla a {\rm d}x.
\]
Then, if $ u $ is a solution to \eqref{INLC}, we have the following identity
\begin{align}
\frac{\rm d}{{\rm d}t} Z(t)
= & ~ \left( \frac4{p+1} - 2 \right) \int_{\R^3} \frac{|u|^{p+1}}{|x|^b} \Delta a {\rm d}x - \frac{4b}{p+1} \int_{\R^3} \frac{|u|^{p+1}}{|x|^{b+2}} x \cdot \nabla a {\rm d}x \nonumber \\
& ~ - \int_{\R^3} |u|^2 \Delta \Delta a {\rm d}x + 4 \Re \int_{\R^3} a_{i,j} \overline{u_i} u_j {\rm d}x - 2 \int_{\R^3} \nabla a \cdot \nabla V |u|^2 {\rm d}x.
\end{align}
\end{lemma}

\begin{proof}
	We omit the proof of Lemma \ref{VMi} here for it is classical and standard.
\end{proof}

Let $ \phi : [0, \infty) \rightarrow [0, 1] $ be a smooth function satisfying
\begin{eqnarray}
\phi(r)=
\begin{cases}
1, & ~ r \leq \frac12, \\
0, & ~ r \ge 1.
\end{cases}
\label{cutoff-function}
\end{eqnarray}
Define $ \omega : [0, \infty) \rightarrow [0, \infty) $ by
\[
\omega(r) := \int_0^r \int_0^s \phi(t) {\rm d}t{\rm d}s.
\]
For given $ R > 0 $, we define a radial function
\begin{align*}
a(r) := R^2 \omega \Big( \frac{r}R \Big), \qquad r = |x|.
\end{align*}
It is easy to see $ a(x) \sim |x|^2 $ for $ |x| \leq \frac{R}2 $.
Moreover, we have the following properties:
\[
0 \leq a^{\prime\prime}(r) \leq 2, ~ 0 \leq \Delta a \leq 10, \qquad \forall ~ r \geq 0, ~ \forall ~ x \in \R^3.
\]

\begin{lemma}[Morawetz estimate]\label{de}
For $ 0 < b < 1 $ and $ 1 + \frac{4-2b}3 < p < 1 + 4 - 2b $ in \eqref{INLC}, let $ V : \R^3 \rightarrow \R $ satisfy \eqref{ASU1}, $ V \ge 0 $, $ x \cdot \nabla V \leq 0 $.
If $ u_0 \in H^1(\R^3) $ satisfy \eqref{q1}, \eqref{q2}, and \eqref{q3}, then for any $ T \ge 0 $ and any $ R \ge \bar{R} $ as in Lemma \ref{lc}, the corresponding global solution to the focusing problem \eqref{INLC} satisfies
\begin{align}
\frac1T \int_0^T \int_{|x| \leq \frac{R}2} \frac{|u(t,x)|^{p+1}}{|x|^b} {\rm d}x {\rm d}t
\lesssim_{u_0, Q} \frac{R}T + \frac1{R^b} + o_R(1).
\end{align}
	
\begin{proof}
Using the Cauchy-Schwarz inequality and the definition of $ Z(t) $, we have
\begin{align}
\sup_{t \in \R} |Z(t)| \lesssim R.
\label{zb}
\end{align}
Denote the angular derivative as
\[
\not\nabla u = \nabla u - \frac{x \cdot \nabla u}{|x|^2} x.
\]
Here, the angular derivative is not necessarily zero, since the solution is not radial.
So, we have
\begin{align}
\frac{\rm d}{{\rm d}t} Z(t)
= & ~ 8 \left[ \int_{|x| \leq \frac{R}2} |\nabla u|^2 + \left( \frac{3-b}{p+1} - \frac32 \right) \int_{|x| \leq \frac{R}2} \frac{|u|^{p+1}}{|x|^b} {\rm d}x \right] \nonumber \\
& ~ + \int_{|x| \geq \frac{R}2} \left[ \left( \frac4{p+1} - 2 \right) \Delta a - \frac{4b}{1+p} \frac{x \cdot \nabla a}{|x|^2} \right] \frac{|u|^{p+1}}{|x|^b} {\rm d}x \nonumber \\
& ~ + 4 \int_{|x| \geq \frac{R}2} \pa^2_r a |\pa_r u|^2 {\rm d}x + 4 \int_{|x| \geq \frac{R}2} \frac{\pa_r a}{|x|} |\not \nabla u|^2 {\rm d}x \nonumber \\
& ~ - 4 \int_{|x| \leq \frac{R}2} x \cdot \nabla V |u|^2 {\rm d}x - 2 \int_{|x| \geq \frac{R}2} \nabla a \cdot \nabla V |u|^2 {\rm d}x.
\end{align}
Observing that $ \| \Delta \Delta a \|_{L^{\infty}} \lesssim \frac1{R^2} $, so we have
\begin{align}
\frac{\rm d}{{\rm d}t} Z(t)
\geq & ~ 8 \left[ \int_{|x| \leq \frac{R}2} |\nabla u|^2 {\rm d}x + \left( \frac{3-b}{p+1} - \frac32 \right) \int_{|x| \leq \frac{R}2} \frac{|u|^{p+1}}{|x|^b} {\rm d}x \right] \nonumber \\
& ~ - \frac{c E^{\frac{p+1}2}}{R^b} - \frac{c}{R^2} M[u_0] - 2 \int_{|x| \geq \frac{R}2} \nabla a \cdot \nabla V |u(t, x)|^2 {\rm d}x.
\label{es7}
\end{align}

Define a smooth cut-off function satisfies
\begin{eqnarray}
\chi^{\rho}(x) =
\begin{cases}
1, & ~ |x| \leq \frac12, \\
0, & ~ |x| \ge \frac12 + \frac1{\rho}.
\end{cases}
\end{eqnarray}
It is not difficult to observe the identity as follow:
\begin{align}
& ~ \int_{|x| \leq \frac{R}2} | \nabla u |^2 {\rm d}x + \left( \frac{3-b}{p+1} - \frac32 \right) \int_{|x| \leq \frac{R}2} \frac{|u|^{p+1}}{|x|^b} {\rm d}x \nonumber \\
= & ~ \left[ \int_{\R^3} \left( \chi_R^{\rho} \right)^2 | \nabla u |^2 {\rm d}x + \left( \frac{3-b}{p+1} - \frac32 \right) \int_{\R^3} \left( \chi_R^{\rho} \right)^2 \frac{|u|^{p+1}}{|x|^b} {\rm d}x \right] \nonumber \\
& ~ - \left[ \int_{\frac{R}2 < |x| \leq \frac{R}2 + \frac{R}{\rho}} \left( \chi_R^{\rho} \right)^2 |\nabla u|^2 {\rm d}x + \left( \frac{3-b}{p+1} - \frac32 \right) \int_{\frac{R}2 < |x| \leq \frac{R}2 + \frac{R}{\rho}} \left( \chi_R^{\rho} \right)^2 \frac{|u|^{p+1}}{|x|^b} {\rm d}x \right] \nonumber \\
= & ~ \left[ \int_{\R^3} \left| \chi_R^{\rho} \nabla u \right|^2 {\rm d}x + \left( \frac{3-b}{p+1} - \frac32 \right) \int_{\R^3} \frac{\left| \chi_R^{\rho} u \right|^{p+1}}{\vert x \vert^b} {\rm d}x \right] \nonumber \\
& ~ - \left[ \int_{\frac{R}2 < |x| \leq \frac{R}2 + \frac{R}{\rho}} \left( \chi_R^{\rho} \right)^2 |\nabla u|^2 {\rm d}x + \left( \frac{3-b}{p+1} - \frac32 \right) \int_{\frac{R}2 < |x| \leq \frac{R}2 + \frac{R}{\rho}} \left( \chi_R^{\rho} \right)^2 \frac{|u|^{p+1}}{|x|^b} {\rm d}x \right] \nonumber \\
& ~ - \left( \frac32 - \frac{3-b}{p+1} \right) \int_{\R^3} \left( \left( \chi_R^{\rho} \right)^{p+1} - \left( \chi_R^{\rho} \right)^2 \right) \frac{|u|^{p+1}}{|x|^b} {\rm d}x \nonumber \\
=: & ~ \left[ \int_{\R^3} \left| \chi_R^{\rho} \nabla u \right|^2 {\rm d}x + \left( \frac{3-b}{p+1} - \frac32 \right) \int_{\R^3} \frac{\left| \chi_R^{\rho} u \right|^{p+1}}{|x|^b} {\rm d}x \right] - I_{\rho} - II_{\rho}.
\label{es8}
\end{align}

According to Lemma \ref{ge}, we get
\begin{align}
& ~ \int_{\R^3} \left| \chi_{R}^{\rho} \nabla u \right|^{2} {\rm d}x + \left( \frac{3-b}{p+1} - \frac32 \right) \int_{\R^3} \frac{\left| \chi_R^{\rho} u \right|^{p+1}}{|x|^b} {\rm d}x \nonumber \\
\ge & ~ \int_{\R^3} \left| \nabla \left( \chi_{R}^{\rho} u \right) \right|^2 {\rm d}x + \left( \frac{3-b}{p+1} - \frac32 \right) \int_{\R^3} \frac{\left| \chi_R^{\rho} u \right|^{p+1}}{|x|^b} {\rm d}x - \frac{c}{R^2} M \left[ u_0 \right].
\label{es9}
\end{align}
The inequalities \eqref{es7}, \eqref{es8} and \eqref{es9} can be rewritten as
\begin{align}
\frac{\rm d}{{\rm d}t} Z(t)
\geq & ~ 8 \left[ \int_{\R^3} \left| \nabla \left( \chi_R^{\rho} u \right) \right|^2 {\rm d}x + \left( \frac{3-b}{p+1} - \frac32 \right) \int_{\R^3} \frac{\left| \chi_R^{\rho} u \right|^{p+1}}{|x|^b} {\rm d}x \right] \nonumber \\
& ~ -\frac{c E^{\frac{p+1}2}}{R^b} - \frac{c}{R^2} M \left[ u_0 \right] - 8 I_{\rho} - 8 II_{\rho} - 2 \int_{|x| \geq \frac{R}2} \nabla a \cdot \nabla V |u|^2 {\rm d}x.
\label{es10}
\end{align}
Noting the fact
\[
\nabla a = R \frac{x}{|x|} \omega^{\prime} \Big( \frac{r}R \Big), \qquad \text{when} ~ |x| \geq \frac{R}2,
\]
and the derivative of $ \omega $ is bounded which is independent of $ R $, by Sobelev embedding and interpolation, we obtain
\[
\| u \|_{L^m}
\lesssim \| u \|^{\vartheta}_{L^2} \| u \|_{L^6}^{1-\vartheta}
\lesssim \| u \|_{H^1}, \qquad 2 \le m \le 6.
\]
It is easy to see that $ 2 \le 2 r^{\prime} \le 6 $ when $ r \ge \frac32 $.
Therefore,
\begin{align}
\left| \int_{|x| \ge \frac{R}2} \nabla a \cdot \nabla V |u(t, x)|^2 {\rm d}x \right|
\lesssim & ~ \int_{|x|\ge \frac{R}2} | \nabla a \cdot \nabla V | |u(t, x)|^2 {\rm d}x \nonumber \\
\lesssim & ~ \int_{\R^3} | x \cdot \nabla V | |u(t, x)|^2 {\rm d}x \nonumber \\
\lesssim & ~ \| x \cdot \nabla V \|_{L_x^r(\R^3)} \| u(t, x) \|_{L_x^{2r^{\prime}}}^2 \nonumber \\
< & ~ \infty, \qquad \forall ~ t \in \R.
\label{es11}
\end{align}
Hence,
\begin{align}
\lim_{R \to \infty} \sup_{t \in \R} \left| \int_{|x| \ge \frac{R}2} \nabla a \cdot \nabla V |u(t, x)|^2 {\rm d}x \right| = 0.
\end{align}
Using Lemma \ref{lc} and recalling that $ 0 < b < 1 $, we can write \eqref{es10} as
\begin{equation*}
\int_{\R^3} \frac{\left| \chi_{R}^{\rho} u(t, x) \right|^{p+1}}{|x|^b} {\rm d}x
\lesssim_{u_0, Q} \frac{\rm d}{{\rm d}t} Z(t) + \frac1{R^b} + 8 I_{\rho} + 8 II_{\rho} + o_R(1).
\end{equation*}

Now, by Dominated Convergence Theorem, we obtain $ I_{\rho} + II_{\rho} \rightarrow 0 $ as $ \rho \rightarrow +\infty $.
Hence,
\begin{equation*}
\int_{|x| \leq \frac{R}2} \frac{|u(t, x)|^{p+1}}{|x|^b} {\rm d}x \lesssim_{u_0, Q} \frac{\rm d}{{\rm d}t} Z(t) + \frac1{R^b} + o_R(1).
\end{equation*}
Then, integrating over time and using \eqref{zb}, we have
\begin{align*}
\frac{1}{T} \int_{0}^{T} \int_{|x| \leq \frac{R}{2}} \frac{|u(t, x)|^{p+1}}{|x|^b} {\rm d}x {\rm d}t \lesssim_{u_0,Q} & ~ \frac{1}{T} \sup _{t \in[0, T]} |Z(t)| + \frac1{R^{b}} + o_R(1) \\
\lesssim_{u_0, Q} & ~ \frac{R}T + \frac1{R^b} + o_R(1).
\end{align*}
Therefore, the proof is completed.
\end{proof}
		
\end{lemma}
\begin{lemma}[Energy evacuation]\label{ee}
For $ 0 < b < 1 $ and $ 1 + \frac{4-2b}3 < p < 1 + 4 - 2b $ in \eqref{INLC}, let $ V : \R^3 \rightarrow \R $ satisfy \eqref{ASU1}, $ V \ge 0 $, $ x \cdot \nabla V \leq 0 $.
If $ u_0 \in H^1(\R^3) $ satisfy \eqref{q1}, \eqref{q2}, and \eqref{q3}, then there exist a sequence of times $ t_n \rightarrow \infty $ and a sequence of  $ R_n \rightarrow \infty $ such that
\begin{equation}
\lim_{n \to \infty} \int_{|x| \le R_n} \frac{|u(t_n)|^{p+1}}{|x|^b} {\rm d}x = 0.
\label{lim=0}
\end{equation}
\end{lemma}

\begin{proof}
Choosing $ T_n = R^3_n $ and applying Lemma \ref{de}, we have
\begin{equation*}
\frac{1}{R_n^3} \int_{0}^{R^3_n} \int_{|x| \leq \frac{R_n}{2}} \frac{|u(t)|^{p+1}}{|x|^b} {\rm d}x {\rm d}t
\lesssim_{u_0,Q} \frac{1}{R^2_n} + \frac{1}{R_n^{b}} + o_R(1).
\end{equation*}
According to the Mean Value Theorem, there exist sequences $ t_n \rightarrow \infty $ and $ R_n \rightarrow \infty $ such that \eqref{lim=0} holds.
\end{proof}

\begin{proof}[Proof of Theorem \ref{MR}]	
Take $ t_n \rightarrow \infty $, and a sequence of $ R_n \rightarrow \infty $ as in Lemma \ref{ee}.
Fix $ \epsilon \ge 0 $ and $ R \ge 0 $ as in Theorem \ref{SC}.
Choosing $ n $ sufficiently large, such that $ R_n \ge R $, by H\"older's inequality yields
\[
\int_{|x| \leq R} |u(x,t_n)|^2 {\rm d}x 
\lesssim R^{\frac{2b+3(p-1)}{p+1}} \left( \int_{|x|\leq R_n} \frac{|u(t_n)|^{p+1}}{\vert x \vert^b} {\rm d}x \right)^{\frac{2}{p+1}}
\rightarrow 0
\]
as $ n \rightarrow \infty $.
Therefore, by Theorem \ref{SC}, $ u $ scatters forward in time.
\end{proof}

{\bf Acknowledgements:} {\rm We are grateful to Professor Yi Zhou for his guidance and encouragement, which greatly support our original manuscript.
Moreover, the authors also would like to thank Professor Deng Zhang for his invaluable discussion and suggestions which helped authors write this paper.}

\begin{center}

\end{center}

\end{document}